\documentclass[11pt]{article}
\usepackage[latin1]{inputenc}
\usepackage{amscd}
\usepackage{amsfonts}
\usepackage{amsmath}
\usepackage{amssymb}
\usepackage{amsthm}

\usepackage{fancyhdr}
\usepackage{graphicx}
\usepackage{hyperref}
\usepackage{indentfirst}
\usepackage{latexsym}
\usepackage{mathrsfs}
\usepackage{xypic}

\usepackage{color}

\usepackage[all]{xy}

\usepackage[top=1in,bottom=1in,left=1.25in,right=1.25in]{geometry}

\def\<{\langle}
\def\>{\rangle}

\def\K{\mathbb{K}}

\newtheorem{df}{Definition}[section]
\newtheorem{thm}{Theorem}[section]
\newtheorem{cor}{Corollary}[section]
\newtheorem{rem}{Remark}[section]

\newtheorem{prop}{Proposition}[section]
\newtheorem{exa}{Example}[section]

\begin{document}
\date{}
\title{(Non-BiHom-Commutative) BiHom-Poisson algebras}
\author{{ Hadjer  Adimi$^{1}$, Hanene  Amri$^{2}$, Sami  Mabrouk$^{3}$, Abdenacer  Makhlouf$^{4}$ }\\
{\small 1. Universit\'e Mohamed El Bachir El Ibrahimi de Bordj Bou Arr\'eridj El-Anasser, 34030 - Algeria,}  \\
{\small E-mail: h.adimi@univ-bba.dz}\\
{\small 2. Universit\'e Badji Mokhtar Annaba' BP 12 d\'epartement de math\'ematiques, facult\'e des sciences, Algeria }\\
{\small E-mail: akian296@yahoo.fr}\\
{\small 3.  University of Gafsa, Faculty of Sciences Gafsa, 2112 Gafsa, Tunisia}\\
{\small E-mail: mabrouksami00@yahoo.fr}\\
{\small 4.~ IRIMAS - D\'epartement de Math\'ematiques, 6, rue des fr\`eres Lumi\`ere,
F-68093 Mulhouse, France}\\ 
{\small E-mail: Abdenacer.Makhlouf@uha.fr}}
 \maketitle{}

 \begin{abstract}
 The aim of this paper is to  introduce and study BiHom-Poisson algebras,  in particular Non-BiHom-Commutative BiHom-Poisson algebras. We discuss their  representation theory and Semi-direct product.  Furthermore, we characterize admissible BiHom-Poisson algebras. Finally, we establish the classification of 2-dimensional BiHom-Poisson algebras.
\end{abstract}
\maketitle
\section{Introduction}
 A vector space $A$ is called a Poisson algebra provided that, beside addition, it has two $\mathbb{K}$-bilinear
operations which are related by derivation. First, with respect to multiplication, $A$
is a commutative associative algebra; denote the multiplication by $\mu(a,b)$ (or $a\cdot b$ or $ab$), where $a, b \in A$. Second, $A$ is a Lie algebra; traditionally here the Lie operation
is denoted by the Poisson brackets \{a,b\}, where $a, b \in A$. It is also assumed that
these two operations are connected by the Leibniz rule
$\{a\cdot b,c\} $= $a \cdot \{b,c\} +b \cdot \{a, c\}$, $a$, $b$, $c \in A$ \cite{gr,kubo}. Poisson algebras are the key to recover Hamiltonian mechanics and are also central in the study of quantum groups. Manifolds with a Poisson algebra structure are known as Poisson manifolds, of which the symplectic manifolds and the Poisson-Lie groups are a special case.  Their  generalization is  known as Nambu algebras \cite{n:generalizedmech,f:nliealgebras,GDito,GDitoF}, where the binary bracket is generalized to ternary or $n$-ary bracket.
A Hom-algebra structure is a multiplication on a vector space where  a usual  structure is twisted by a homomorphism \cite{ms}. The first motivation to study nonassociative Hom-algebras comes from quasi-deformations of Lie algebras of
vector fields, in particular $q$-deformations of Witt and Virasoro algebras. The structure of  (Non-Commutative)-Hom-Poisson algebras are twisted generalization of (Non-Commutative)-Poisson algebras \cite{Yau:Noncomm}. A (Non-Commutative)-Hom-Poisson algebra $A$ is defined by a  linear self-map $\alpha$, and two binary operations $\{, \}$ (the Hom-Lie bracket) and $\mu$ (the Hom-associative product). The associativity, the Jacobi identity, and the Leibniz identity in a (Non-Commutative)-Poisson algebra are replaced by their Hom-type (i.e. $\alpha$-twisted) identities. Motivated by a categorical study of Hom-algebras and new type of categories,    generalized algebraic structures endowed with two commuting multiplicative linear maps, called BiHom-algebras including BiHom-associative algebras, BiHom-Lie algebras and BiHom-Bialgebras were introduced in \cite{GrazianiMakhloufMeniniPanaite}. Therefore,   when the two linear maps are the same, BiHom-algebras will be turn to Hom-algebras in some cases. Various studies deal with these new  type of algebras, see \cite{RepBiHomLie,luimakhlouf,luimakhlouf2} and references therein. 

The purpose of this paper is to study  (Non-Commutative) BiHom-Poisson algebras. The paper is organized as follows. In Section 2, we review the definition of BiHom-associative and BiHom-Lie algebras and then  generalize the Poisson algebra notion to BiHom case. This new structure is  illustrated with some examples. In Section 3, we study the concept of module of BiHom-Poisson algebra, which is based on  BiHom-modules of BiHom-associative and BiHom-Lie algebras.   Then we define semi-direct product of (Non-Commutative) BiHom-Poisson algebras. In Section 4, we describe  BiHom-Poisson algebras using only one binary operation and the twisting maps via the polarization-depolarization process. We show that, admissible BiHom-Poisson algebras, and only these BiHom-algebras, give rise to BiHom-Poisson algebras via polarization. In the last section, we give the classification of 2-dimensional BiHom-Poisson algebras.

\section{Definitions and Examples }
In this section, we recall some basic definitions about BiHom-associative and  BiHom-Lie algebras \cite{GrazianiMakhloufMeniniPanaite} and then we generalize  the Poisson algebras notion to BiHom case. We assume
 that $\K$ will denote a commutative field of characteristic zero.
 \begin{df}
   A BiHom-associative algebra is a quadruple $(A,\mu,\alpha,\beta)$ consisting of vector space $A$, a bilinear mapping
$\mu:A\times A\rightarrow A$ and two  homomorphisms $\alpha,\beta:A\rightarrow A$ such that for   $x,y,z\in A$ we have
\begin{eqnarray}
&& \alpha  \beta = \beta  \alpha,\nonumber\\
&& \alpha \circ\mu=\mu\circ\alpha^{\otimes^2},~~\beta \circ\mu=\mu\circ\beta^{\otimes^2},\nonumber\\
&&\mu(\alpha(x),\mu(y,z))=\mu(\mu(x,y),\beta(z))~~(\textrm{BiHom-associative\ condition}),\label{Bihom associative}
\end{eqnarray}
where $\alpha\beta =\alpha\circ\beta$.
 \end{df}
 We recall that the BiHom-commutative condition is $\mu(\beta(x),\alpha(y))=\mu(\beta(y),\alpha(x))$, for all $x,y \in A$.
\begin{df}    A  BiHom-Lie  algebra is a quadruple $(A,[\cdot,\cdot],\alpha,\beta)$ consisting of vector space $A$, a bilinear mapping
$[.,.]:A\times A\rightarrow A$ and two  homomorphisms $\alpha,\beta:A\rightarrow A$ such that for   $x,y,z\in A$ we have
\begin{eqnarray}
&& \alpha  \beta = \beta  \alpha,\nonumber\\
&& \alpha ([x,y])=[\alpha(x),\alpha(y)],~~\beta ([x,y])=[\beta(x),\beta(y)],\nonumber\\
&& [\beta(x),\alpha(y)]=-[\beta(y),\alpha(x)], \ (\textrm{BiHom-skew-symmetric})\label{anti-biho}\\
&& \circlearrowleft_{x,y,z}[\beta^2(x),[\beta(y),\alpha(z)]]= 0~~(\textrm{BiHom-Jacobi\ condition}),\label{Bihom jacobi}
\end{eqnarray}
where $\circlearrowleft_{x,y,z}$ denotes summation over the cyclic permutation on $x,y,z$.

\end{df}
If $\alpha$ is a bijective morphism, then the identity \eqref{Bihom jacobi} can be written \begin{equation}\label{Bihom jacobi1}
  [\beta^2(x),[\beta(y),\alpha(z)]]=[[\alpha^{-1}\beta^2(x),\beta(y)],\alpha\beta(z)]+[\beta^2(y),[\beta(x),\alpha(z)]].
\end{equation}

\begin{df} A Poisson algebra is a triple $(A, \{\cdot,\cdot\}, \mu)$ consisting of a vector space $A$ and two bilinear maps $\{\cdot,\cdot\},\ \mu : A\times A \longrightarrow A$  satisfying
\begin{enumerate}
\item $(A, \{\cdot,\cdot\})$ is a Lie algebra,
\item $(A, \mu)$ is a commutative associative algebra,
\item for all $x, y \in A$ :
\begin{equation}\label{a}
\{\mu(x,y),z\} = \mu(\{x, z\},y)+\mu(x, \{y,z\})\ (\textrm{Compatibility\ identity}).
\end{equation}
\end{enumerate}
If $\mu$ is non-commutative then  $(A, \{\cdot,\cdot\}, \mu)$ is a non-commutative Poisson algebra.
\end{df}
\begin{df} A BiHom-Poisson algebra is a 5-uple $(A, \{\cdot,\cdot\}, \mu, \alpha,\beta)$ consisting of a vector space $A$, two bilinear maps $\{\cdot,\cdot\},\ \mu : A\times A \longrightarrow A$  and two  linear maps $\alpha,\ \beta:A \longrightarrow A$ satisfying
\begin{enumerate}
\item $(A, \{\cdot,\cdot\},\alpha,\beta)$ is a BiHom-Lie algebra,
\item $(A, \mu, \alpha,\beta)$ is a BiHom-commutative BiHom-associative algebra,
\item for all $x, y \in A$ :
\begin{equation}\label{a}
\{\mu(x,y),\alpha\beta(z)\} = \mu(\{x, \beta(z)\},\alpha(y))+\mu(\alpha(x), \{y,\alpha(z)\}).
\end{equation}
\end{enumerate}
If $\mu$ is non-BiHom-commutative then  $(A, \{\cdot,\cdot\}, \mu, \alpha,\beta)$ is a non-BiHom-commutative BiHom-Poisson algebra.
\end{df}

We are using here a right handed Leibniz rule, one may call such algebras right BiHom-Poisson algebras. We refer to \cite{LMMP20} for left BiHom-Poisson algebras.

\begin{rem}
  Obviously, a BiHom-Poisson  algebra $(A, \{\cdot,\cdot\}, \mu, \alpha,\beta)$ for which $\alpha=\beta$  and $\alpha$ injective is just a
Hom-Poisson  algebra $(A, \{\cdot,\cdot\}, \mu, \alpha)$.
\end{rem}

\begin{prop} Let $(A, \mu, \alpha,\beta)$ be a BiHom-associative algebra where  $\alpha$ and $\beta$ are two bijective homomorphisms. Then the $5$-uple $(A, \{\cdot,\cdot\},\mu,\alpha,\beta)$, where the bracket is defined by
$$\{x,y\}=\mu(x,y)-\mu(\alpha^{-1}\beta(y),\alpha\beta^{-1}(x)),$$
for $x, y \in A$
is a non-commutative  BiHom-Poisson  algebra.
\end{prop}
\begin{proof}We show that $\alpha$ and $\beta$ are compatible with the bracket $\{\cdot,\cdot\}$ . For all $x, y \in A$, we have
\begin{eqnarray*}
\{\alpha(x),\alpha(y)\}&=& \mu(\alpha(x),\alpha(y))-\mu(\alpha^{-1}\beta(\alpha(y)),\alpha\beta^{-1}(\alpha(x)))\\
&=& \mu(\alpha(x),\alpha(y))-\mu(\beta(y),\alpha^2\beta^{-1}(x))\\
&=& \alpha(\{x,y\}).
\end{eqnarray*}
The second equality holds since $\alpha$ is even and $\alpha\circ \beta=\beta\circ \alpha$. In the same way, we check that
$\beta(\{x, y\})= \{\beta(x),\beta(y)\}$.\\
The BiHom-skew-symmetry $\{\beta(x),\alpha(y)\}=-\{\beta(y),\alpha(x)\}$ is obvious.\\
Therefore, it remains to prove the BiHom-Jacobi identity. For all $x, y,z \in A$, we have
\begin{eqnarray*}
\{\beta^2(x),\{\beta(y),\alpha(z)\}\}&=&  \mu(\beta^2(x),\mu(\beta(y),\alpha(z)))-\mu(\mu(\alpha^{-1}\beta^2(y),\beta(z)),\alpha\beta(x)) \\
&-& \mu( \beta^2(x),\mu(\beta(z),\alpha(y)))+\mu(\mu(\alpha^{-1}\beta^2(z),\beta(y)),\alpha\beta(x)).
\end{eqnarray*}
And, we have
\begin{eqnarray*}
\{\beta^2(y),\{\beta(z),\alpha(x)\}\}
&=&\mu(\beta^2(y),\mu(\beta(z),\alpha(x)))-\mu(\mu(\alpha^{-1}\beta^2(z),\beta(x)),\alpha\beta(y)) \\
&-& \mu( \beta^2(y),\mu(\beta(x),\alpha(z)))+\mu(\mu(\alpha^{-1}\beta^2(x),\beta(z)),\alpha\beta(y)) .
\end{eqnarray*}
Similarly,
\begin{eqnarray*}
\{\beta^2(z),\{\beta(x),\alpha(y)\}\}
&=& \mu(\beta^2(z),\mu(\beta(x),\alpha(y)))-\mu(\mu(\alpha^{-1}\beta^2(x),\beta(y)),\alpha\beta(z)) \\
&-& \mu( \beta^2(z),\mu(\beta(y),\alpha(x)))+ \mu(\mu(\alpha^{-1}\beta^2(y),\beta(x)),\alpha\beta(z)).
\end{eqnarray*}
By  BiHom-associativity, we find that
\begin{eqnarray*}
\circlearrowleft_{x,y,z}\{\beta^2(x),\{\beta(y),\alpha(z)\}\}&=& 0.
\end{eqnarray*}
Now, we show the compatibility condition, for $x,y,z\in P$, we have
\begin{align*}
 & \{\mu(x,y),\alpha\beta(z)\} -\mu(\{x, \beta(z)\},\alpha(y))-\mu(\alpha(x) \{y,\alpha(z)\})  \\
  = &\mu(\mu(x,y),\alpha\beta(z))-\mu(\beta^2(z),\mu(\alpha\beta^{-1}(x),\alpha\beta^{-1}(y)) -\mu(\mu(x, \beta(z)),\alpha(y))\\&+\mu(\mu(\alpha^{-1}\beta^2(z),\alpha\beta^{-1}(x)),\alpha(y))  -\mu(\alpha(x),\mu(y,\alpha(z)))+\mu(\alpha(x),\mu(\beta^2(z),\alpha\beta^{-1}(y)))=0.
\end{align*}
\end{proof}

\begin{df}
Let $(A,\mu,\{.,.\},\alpha,\beta)$ and $(A',\mu',\{.,.\}',\alpha',\beta')$ be two  BiHom-Poisson  algebras.
A linear map $f : A\rightarrow A'$ is a \emph{morphism} of BiHom-Poisson algebras if it satisfies for all $x_1,x_2\in A$:
\begin{eqnarray}
f(\{x_{1},x_{2}\}) &=& \{f(x_{1}),f(x_{2})\}' ,\\
f \circ \mu&=&\mu'\circ f^{\otimes 2},\\
f\circ \alpha &= & \alpha' \circ f.\\
f\circ \beta &= & \beta' \circ f.
\end{eqnarray}
It said to be a  \emph{weak morphism} if    hold only  the two first conditions.
\end{df}

\begin{df}
Let  $(A,\mu,\{.,.\},\alpha,\beta)$ be a BiHom-Poisson algebra.
It  is said to be  \emph{multiplicative} if
\begin{eqnarray*}
\alpha(\{x_{1},x_{2}\})&=&\{\alpha(x_{1}),\alpha(x_{2})\},\\
 \beta(\{x_{1},x_{2}\})&=&\{\beta(x_{1}),\beta(x_{2})\},\\
 \alpha \circ \mu&=&\mu \circ \alpha^{\otimes 2}.\\
 \beta \circ \mu&=&\mu \circ \beta^{\otimes 2}.
\end{eqnarray*}
It  is said to be  \emph{regular} if $\alpha$ and $\beta$ are bijective.
\end{df}

\begin{prop}\label{twist}
  Let $(A, \{\cdot,\cdot\}, \mu)$  be an ordinary Poisson algebra over a field $\K$ and let
$\alpha,\beta: A\rightarrow A$ be two commuting morphisms. Define the two  linear maps $\{\cdot,\cdot\}_{ \alpha,\beta},\mu_{ \alpha,\beta}:A\otimes A\longrightarrow A$ by
$$\{x,y\}_{ \alpha,\beta}=\{ \alpha(x),\beta(y)\}\ \textrm{and} \ \mu_{ \alpha,\beta}(x,y)=\mu (\alpha(x),\beta(y)),$$ for all $x,y\in A$.\\
Then $A_{ \alpha,\beta}
:=(A, \{\cdot,\cdot\}_{ \alpha,\beta}, \mu_{ \alpha,\beta} ,\alpha,\beta)$ is a BiHom-Poisson algebra.
\end{prop}
\begin{proof}
We already have $(A, \mu, \alpha,\beta)$ is a BiHom-commutative BiHom-associative algebra and $(A, \{ , \}_{\alpha,\beta}, \alpha,\beta)$ is a BiHom-Lie algebra. It remains to check the BiHom-Leibniz identity. Let $x,\ y,\ z\in A$, we have
\begin{align*}
 & \{\mu_{ \alpha,\beta}(x,y),\alpha\beta(z)\}_{ \alpha,\beta} - \mu_{ \alpha,\beta}(\{x,\beta(z)\}_{ \alpha,\beta},\alpha(y))-\mu_{ \alpha,\beta}(\alpha(x), \{y,\alpha(z)\}_{ \alpha,\beta})\\
  &= \{\mu(\alpha^2(x),\alpha\beta(y)),\alpha\beta^2(z)\} - \mu(\{\alpha^2(x), \alpha\beta^2(z)\},\alpha\beta(y))-\mu(\alpha^2(x), \{\alpha\beta(y),\alpha\beta^2(z)\}) \\
   & = \{\mu(X,Y),Z)\} - \mu(\{X, Z\},Y)-\mu(X, \{Y,Z\})=0,
\end{align*}
where $X=\alpha^2(x),\ Y=\alpha\beta(y),\ Z=\alpha\beta^2(z).$
\end{proof}
\begin{rem}
  Let $(A, \{\cdot,\cdot\}, \mu,\alpha,\beta)$ be a  BiHom-Poisson algebra and $\alpha', \beta': A \to A$ two  BiHom-Poisson algebra morphisms such that any of the maps $\alpha, \beta, \alpha', \beta'$ commute. Define new multiplications on $A$ by:
 \begin{align*}
    & \{x,y\}'= \{\alpha'(x) , \beta'(y)\}, \quad
    \mu'(x  ,y)=\mu( \alpha'(x) , \beta'(y)).
 \end{align*}
 Then, $(A, \{\cdot,\cdot\}', \mu',\alpha'\alpha,\beta'\beta)$ is a BiHom-Poisson algebra.
\end{rem}
\begin{exa}\label{example1HomPoisson}
Let $\{e_1,e_2,e_3\}$  be a basis of a $3$-dimensional vector space
$A$ over $\K$.  Consider the following multiplication $\mu$, skew-symmetric
bracket and linear map $\alpha$ on $A-
=\K^3${\rm :}
\[
\begin{array}{ll}
 \begin{array}{lll}
 \mu(e_1,e_1) &=&  e_1, \ \\
 \mu(e_1,e_2) &=& \mu(e_2,e_1)=e_3,\\
 \end{array}
 & \quad
 \begin{array}{lll}
 \{ e_1,e_2 \}&=& a e_2+ b e_3, \ \\
 \{e_1, e_3 \}&=& c e_2+ d e_3, \ \\
 \end{array}
\end{array}
\]

\[
\alpha (e_1)= \lambda_1 e_2+\lambda_2 e_3 , \quad
\alpha (e_2) =\lambda_3 e_2+\lambda_4 e_3  , \quad
\alpha (e_3)=\lambda_5 e_2+\lambda_6 e_3,
\]
where $a,b,c,d,\lambda_1,\lambda_2,\lambda_3,\lambda_4,\lambda_5,\lambda_6$ are parameters
in $\K$.
Assume that $\beta=Id$, hence $\alpha\beta=\beta\alpha$.  Using Proposition \ref{twist}, we construct the following  multiplicative BiHom-Poisson algebra defined by
\[
\begin{array}{ll}
 \begin{array}{lll}
 \mu_{\alpha\beta}(e_1,e_1) &=&  \lambda_{1} e_3, \  \\
 \mu_{\alpha\beta}(e_2,e_1) &=& \lambda_{3} e_3,\\
  \mu_{\alpha\beta}(e_3,e_1) &=& \lambda_{5} e_3,\\
 \end{array}
 & \quad
 \begin{array}{lll}
 \{ e_1,e_1 \}_{\alpha\beta}&=& -(\lambda_{1} a+\lambda_{2} c) e_2 -(\lambda_{1} b+\lambda_{2} d) e_3, \ \\
 \{ e_2,e_1 \}_{\alpha\beta}&=& -(\lambda_{3} a+\lambda_{4} c) e_2 -(\lambda_{3} b+\lambda_{4} d) e_3, \ \\
  \{ e_3,e_1 \}_{\alpha\beta}&=& -(\lambda_{5} a+\lambda_{6} c) e_2 -(\lambda_{5} b+\lambda_{6} d) e_3. \ \\
 \end{array}
\end{array}
\]
\end{exa}
Then we give an example of BiHom-Poisson algebra where $\alpha$ and $\beta$ are arbitrary and $\{e_1, e_2, e_3\}$ be a basis of a 3-dimensional vector space $A$ over $\K$.
\begin{exa}\label{example2biHomPoisson}
\[
\alpha (e_1)=  e_2 , \quad
\alpha (e_2) = e_2  , \quad
\alpha (e_3)= e_3.
\]
\[
\beta (e_1)=  e_1 , \quad
\beta (e_2) = e_2 .
\]
\[
\begin{array}{ll}
 \begin{array}{lll}
 \mu(e_1,e_2) &=& \lambda_{1} e_2, \ \\
 \mu(e_2,e_1) &=& \lambda_{1} e_1,\\
 \end{array}
 & \quad
 \begin{array}{lll}
 \{ e_1,e_2 \}= a e_3, \ \\
 \end{array}
\end{array}
\]
where $a,\lambda_1$ are parameters
in $\K$.
\end{exa}
  Another example of BiHom-Poisson algebras of dimension 3 with basis $\{e_1, e_2, e_3\}$ is given where $\alpha$ and $\beta$ are diagonal.
 \begin{exa}\label{example3biHomPoisson}
\[
\alpha (e_1)= a e_1 , \quad
\alpha (e_2) = b e_2 .
\]
\[
\beta (e_1)= c e_1 , \quad
\beta (e_2) = d e_2 .
\]
\[
\begin{array}{ll}
 \begin{array}{lll}
 \mu(e_3,e_3) = \lambda_{1} e_3, \ \\

 \end{array}
 & \quad
 \begin{array}{lll}
 \{ e_3,e_3 \}= \lambda_{2} e_3, \ \\
 \end{array}
\end{array}
\]
where $a,b,c,d,\lambda_1,\lambda_2$ are parameters in $\K$.
\end{exa}

In the sequel we define a direct sum and tensor product of a BiHom-Poisson algebra and a  BiHom-associative symmetric algebra.

\begin{thm}
Let $(A_{1},\mu_{1},\{.,.\}_{1},\alpha_{1},\beta_{1})$ and $(A_{2},\mu_{2},\{.,.\}_{2},\alpha_{2},\beta_{2})$ be two (non-BiHom-commutative) BiHom-Poisson algebras.  Let $\mu_{A_{1}\oplus A_{2}}$ be a bilinear map  on $A_{1}\oplus A_{2}$ defined for $x_1,y_1\in A_1$ and $x_2,y_2\in A_2$ by
$$\mu(x_{1}+x_{2},y_{1}+y_{2})=\mu_{1}(x_{1},y_{1})+\mu_{2}(x_{2},y_{2}),$$   $\{.,.\}_{A_{1}\oplus A_{2}}$  a bilinear map  defined by $$\{x_{1}+x_{2},y_{1}+y_{2}\}_{A_{1}\oplus A_{2}}=\{x_1,y_1\}_1+\{x_2,y_2\}_2$$ and  $\alpha_{A_{1}\oplus A_{2}}$ a linear map defined by $$\alpha_{A_{1}\oplus A_{2}}(x_1+x_2)=\alpha_1(x_1)+\alpha_2(x_2),$$
and  $\beta_{A_{1}\oplus A_{2}}$ a linear map defined by $$\beta_{A_{1}\oplus A_{2}}(x_1+x_2)=\beta_1(x_1)+\beta_2(x_2).$$  Then
$$(A_{1}\oplus A_{2},\mu_{A_{1}\oplus A_{2}},\{.,.\}_{A_{1}\oplus A_{2}},\alpha_{A_{1}\oplus A_{2}},\beta_{A_{1}\oplus A_{2}})$$
is a BiHom-Poisson algebra.
\end{thm}

\begin{thm}
Let $(A,\mu,\{.,.\},\alpha,\beta)$ be a  BiHom-Poisson algebra and  $(B,\mu',\alpha',\beta')$ be a  BiHom-associative symmetric algebra, then
$$(A\otimes B,\{.,.\}_{A\otimes B}, \mu \otimes \mu',\alpha\otimes\alpha',\beta\otimes\beta'),$$
is a BiHom-Poisson algebra, where $\{.,.\}_{A\otimes B}=\{.,.\}\otimes\mu'. $
\end{thm}\begin{proof}
Since $\mu$ and $\mu'$ are both BiHom-associative multiplication whence a tensor product $\mu \otimes \mu'$ is BiHom-associative. Also the BiHom-commutativity of
$\mu \otimes \mu'$, the BiHom-skewsymmetry of $\{.,.\}$ and the BiHom-commutativity of $\mu$ imply the BiHom-skewsymmetry of $\{.,.\}_{A\otimes B}$. Same, since the BiHom-Jacobi identity of $\{.,.\}$ and the BiHom-associative of $\mu'$ are satisfy then $\{.,.\}_{A\otimes B}$ is a BiHom-Lie bracket on $A\otimes B$. Therefore, it  remains
to check  the BiHom-Leibniz identity.
 We have
\begin{align*}
LHS=&\{\mu \otimes\mu'(a_{1}\otimes b_{1},a_{2}\otimes b_{2}),\alpha\beta\otimes\alpha'\beta'(a_{3}\otimes b_{3})\}_{A\otimes B}\\
&=\{\mu(a_{1},b_{1})\otimes \mu'(a_{2},b_{2}),\alpha\beta(a_{3})\otimes\alpha'\beta'(b_{3})\}_{A\otimes B}\\
&=\underbrace{\{\mu(a_{1},b_{1}),\alpha\beta(a_{3})\}_{A}}_{a'}\otimes \underbrace{\mu'(\mu'(a_{2},b_{2}),\alpha'\beta'(b_{3}))}_{b'}\\
\end{align*}
and
\begin{align*}
RHS=&\mu \otimes\mu'(a_{1}\otimes b_{1},\{\beta(a_{2})\otimes \beta'(b_{2}),\alpha(a_{3})\otimes\alpha'( b_{3})\}_{{A\otimes B}})\\
&+\mu \otimes\mu'(\{\alpha(a_{1})\otimes\alpha'( b_{1}),a_{3}\otimes b_{3}\}_{A\otimes B},\alpha \otimes \beta'(a_{2}\otimes b_{2}))\\
&=\mu \otimes\mu'(\alpha (a_{1})\otimes \alpha'(b_{1}),\{a_{2},\alpha(a_{3})\}\otimes \mu'(b_{2},\alpha'(b_{3})))\\
&+\mu \otimes\mu'(\{a_{1},\beta(a_{3})\}\otimes \mu'(b_{1},\beta'(b_{3})),\alpha (a_{2})\otimes \alpha'(b_{2}))\\
&=\underbrace{\mu(\alpha (a_{1}),\{a_{2},\alpha(a_{3})\})}_{c'}\otimes \underbrace{\mu'(\alpha'(b_{1}),\mu'(b_{2},\alpha'(b_{3})))}_{d'}\\
&+\underbrace{\mu(\{a_{1},\beta(a_{3})\},\alpha (a_{2}))}_{e'}\otimes \underbrace{\mu'(\mu'(b_{1},\beta'(b_{3})),\alpha'(b_{2})}_{f'}.
\end{align*}

With BiHom-Leibniz identity we have $a'=c'+e'$, and using the BiHom-associativity  condition
we have $b'=d'=f'$. Therefore the left hand side is equal to the right hand side and the BiHom-Leibniz identity  is proved. Then
\begin{center}
$(A\otimes B, \mu \otimes \mu',\{.,.\}_{A\otimes B},(\alpha\otimes\alpha',\beta \otimes\beta'))$
\end{center}

is a BiHom-Poisson algebra.

\end{proof}

\section{Modules and semi-direct product  of BiHom-Poisson algebras}
In this section we introduce a representation theory of BiHom-Poisson algebras and provide a  semi-direct product construction.

\begin{df} A representation of a BiHom-Lie algebra $(A, \{\cdot,\cdot\}, \alpha,\beta)$ on a vector space $V$ with respect to  two commuting maps $\gamma,\nu\in End(V)$ is a linear map $\rho_{\{\cdot,\cdot\}}:A\longrightarrow End(V)$, such that for any $x, y\in A$, the following equalities are satisfied:
\begin{eqnarray}
&&\rho_{\{\cdot,\cdot\}}(\alpha(x))\circ \gamma=\gamma\circ\rho_{\{\cdot,\cdot\}}(x),\\ &&\ \rho_{\{\cdot,\cdot\}}(\beta(x))\circ \nu=\nu\circ\rho_{\{\cdot,\cdot\}}(x),\\
&&\rho_{\{\cdot,\cdot\}}(\{\beta(x), y\})\circ \nu=(\rho_{\{\cdot,\cdot\}}(\alpha\beta(x))\circ\rho_{\{\cdot,\cdot\}}(y)
-\rho_{\{\cdot,\cdot\}}(\beta(y))\circ\rho_{\{\cdot,\cdot\}}(\alpha(x)).
\end{eqnarray}
\end{df}
\begin{prop}
  Let $(A,\{\cdot,\cdot\})$ be a Lie algebra and $\rho:A\rightarrow End(V)$ be a representation of the Lie algebra on $V$. Let  $\alpha,\beta:A\rightarrow A$ be  two commuting morphisms  and let $\gamma,\nu:V\rightarrow V$ be two commuting linear maps such that  $\rho_{\{\cdot,\cdot\}}(\alpha(x))\circ \gamma=\gamma\circ\rho_{\{\cdot,\cdot\}}(x)$, $ \rho_{\{\cdot,\cdot\}}(\beta(x))\circ \nu=\nu\circ\rho_{\{\cdot,\cdot\}}(x)$ and $ \rho_{\{\cdot,\cdot\}}(\alpha(x))\circ \nu=-\rho_{\{\cdot,\cdot\}}(\beta(x))\circ \gamma$. Define $\widetilde{\rho}_{\{\cdot,\cdot\}}=\rho_{\{\cdot,\cdot\}}(\alpha(x))\circ\gamma.$   Then $(V,\widetilde{\rho},\gamma,\nu)$ is a representation of the BiHom-Lie algebra $A$.
\end{prop}
\begin{proof}
Let $x,y \in A$,
\small{\begin{eqnarray*}
&&\widetilde{\rho}_{\{\cdot,\cdot\}}(\{\beta(x), y\}_{\alpha,\beta})\circ \nu-\widetilde{\rho}_{\{\cdot,\cdot\}}(\alpha\beta(x))\circ\widetilde{\rho}_{\{\cdot,\cdot\}}(y)
\widetilde{\rho}_{\{\cdot,\cdot\}}(\beta(y))\circ\widetilde{\rho}_{\{\cdot,\cdot\}}(\alpha(x))\\&&
=\widetilde{\rho}_{\{\cdot,\cdot\}}({\alpha\beta(x), \beta(y)})\circ \nu-\widetilde{\rho}_{\{\cdot,\cdot\}}(\alpha\beta(x))\circ\rho_{\{\cdot,\cdot\}}(\alpha(y))\circ \nu
+\widetilde{\rho}_{\{\cdot,\cdot\}}(\beta(y))\circ\rho_{\{\cdot,\cdot\}}(\alpha^2(x))\circ \nu=\\&&\rho_{\{\cdot,\cdot\}}({\alpha^2\beta(x), \alpha\beta(y)})\circ \nu^2-\rho_{\{\cdot,\cdot\}}(\alpha^2\beta(x))\circ\rho_{\{\cdot,\cdot\}}(\alpha\beta(y))\circ \nu^2
+\rho_{\{\cdot,\cdot\}}(\alpha\beta(y))\circ\rho_{\{\cdot,\cdot\}}(\alpha^2\beta(x))\circ \nu^2\\&&=
\big(\rho_{\{\cdot,\cdot\}}({\alpha^2\beta(x), \alpha\beta(y)})-\rho_{\{\cdot,\cdot\}}(\alpha^2\beta(x))\circ\rho_{\{\cdot,\cdot\}}(\alpha\beta(y))
+\rho_{\{\cdot,\cdot\}}(\alpha\beta(y))\circ\rho_{\{\cdot,\cdot\}}(\alpha^2\beta(x))\big)\circ \nu^2=0.
\end{eqnarray*}}
\end{proof}
Let $(A,\{\cdot,\cdot\},\alpha,\beta)$ be a BiHom-Lie algebra. Let $\rho_{\{\cdot,\cdot\}}:A\rightarrow End(V)$ be a representation of the BiHom-Lie algebra on $V$ with respect to $\gamma$ and $\nu$.  Assume that the maps $\alpha$ and $\nu$ are bijective. On the direct sum of the underlying vector spaces $A\oplus V$, define $\widetilde{\alpha},\widetilde{\beta}:A\oplus V\longrightarrow A\oplus V$ by
\begin{eqnarray*}
 \widetilde{\alpha}(x+a) &=& \alpha(x)+\gamma(a),\\
 \widetilde{\beta}(x+a) &=& \beta(x)+\nu(a),
\end{eqnarray*}
and define a skewsymmetric bilinear map $[\cdot,\cdot]_{A\oplus V}:A\oplus V \times A\oplus V\longrightarrow A\oplus V$ by
\begin{eqnarray}
 [(x+a),(y+b)]_{A\oplus V} &=& \{x,y\}+\rho_{\{\cdot,\cdot\}}(x)(b)-\rho_{\{\cdot,\cdot\}}(\alpha^{-1}\beta(y))(\gamma\nu^{-1}(a)).
\end{eqnarray}
\begin{thm}\label{LieProduitSDirect}\cite{RepBiHomLie} With the above notations, $(A\oplus V,[\cdot,\cdot]_{A\oplus V},\widetilde{\alpha},\widetilde{\beta})$ is a BiHom-Lie algebra.
\end{thm}

\begin{df}
Let $(A,\mu,\alpha,\beta)$ be a commutative BiHom-associative algebra. A representation (or a BiHom-module) on a vector space $V$ with respect to $\gamma,\nu\in End(V)$ is a linear map $\rho_{\mu}:A\longrightarrow End(V)$, such that for any $x, y\in A$, the following equalities are satisfied:
\begin{eqnarray}
& \rho_{\mu}(\alpha(x))\circ \gamma=\gamma\circ\rho_{\mu}(x),\ \rho_{\mu}(\beta(x))\circ \nu=\nu\circ\rho_{\mu}(x),
\\
& \rho_{\mu}(\mu(x, y))\circ \nu=\rho_{\mu}(\alpha(x))\rho_{\mu}(y).
\end{eqnarray}
\end{df}
Let $(A,\mu,\alpha,\beta)$  be a commutative BiHom-associative algebra and $(V,\rho_\mu,\gamma,\nu)$ be a representation of $A$. On the direct sum of the underlying vector spaces $A\oplus V$, define $\widetilde{\alpha},\widetilde{\beta}:A\oplus V\longrightarrow \mathcal{A}\oplus V$ by $$\widetilde{\alpha}(x+a)=\alpha(x)+\gamma(a)\ \textrm{and}\ \widetilde{\beta}(x+a)=\beta(x)+\nu(a) $$ and define a bilinear map $\mu_{A\oplus V}:A\oplus V \times \mathcal{A}\oplus V\longrightarrow A\oplus V$ by
\begin{eqnarray}
\mu_{A\oplus V}(x+a, y+b)&=& \mu(x, y)+\rho_\mu(x)(b)+\rho_\mu(\alpha^{-1}\beta(y))(\gamma\nu^{-1}(a)).
\end{eqnarray}
\begin{thm}\label{AssProduitSDirect} With the above notations, $(A\oplus V,\mu_{A\oplus V},\widetilde{\alpha},\widetilde{\beta})$ is a commutative BiHom-associative algebra.
\end{thm}
\begin{proof} By the fact that $\alpha,\beta$ are  algebra homomorphisms with respect to $\mu$, for $x,y\in A,\ a,b\in V$, we have
\begin{align*}
         \widetilde{\alpha}( \mu_{\mathcal{A}\oplus V}(x+a, y+b)) & = \widetilde{\alpha}(\mu(x, y)+\rho_\mu(x)(b)+\rho_\mu(\alpha^{-1}\beta(y))(\gamma\nu^{-1}(a)))   \\
            & = \alpha(\mu(x, y))+\gamma(\rho_\mu(x)(b))+\gamma(\rho_\mu(\alpha^{-1}\beta(y))(\gamma\nu^{-1}(a))))  \\
            & \mu(\alpha(x),\alpha( y)))+\rho_\mu(\alpha(x))(\gamma(b)))+\rho_\mu(\alpha^{-1}\beta(\alpha(y)))(\gamma\nu^{-1}(\gamma(a)))))\\
            &=\mu_{\mathcal{A}\oplus V}(\alpha(x)+\gamma(a), \alpha(y)+\gamma(b))\\
            &=\mu_{\mathcal{A}\oplus V}(\widetilde{\alpha}(x+a), \widetilde{\alpha}(y+b)).  \end{align*}
            If  $(\mathcal{A},\mu,\alpha,\beta)$ is a commutative BiHom-associative algebra,
\begin{eqnarray}\label{198a}
\mu_{A\oplus V}(\widetilde{\alpha}(x+a),\mu_{A\oplus V}\textbf{(}(y+b,z+c)\textbf{)})=\mu_{A\oplus V}\textbf{(}\mu_{A\oplus V}(x+a,y+b)\textbf{)},\widetilde{\beta}(z+c)),
\end{eqnarray}
for $x, y, z\in A$ and $a, b, c\in V$. Developing (\ref{198a}), we have
\begin{align*}
 &\mu_{A\oplus V}(\widetilde{\alpha}(x+a),\mu_{A\oplus V}\textbf{(}(y+b,z+c)\textbf{)})\\&=
\mu_{A\oplus V}(\widetilde{\alpha}(x+a),\mu(y,z)+\rho_\mu(y)(c)+\rho_\mu(\alpha^{-1}\beta(z))(\gamma\nu^{-1}(b)))\\
&=\mu(\alpha(x),\mu(y,z))+
\rho_\mu(\alpha(x))\circ\rho_\mu(y)(c))
+\rho_\mu(\alpha(x))\circ\rho_\mu(\alpha^{-1}\beta(z))(\gamma\nu^{-1}(b)))\\
&+\rho_\mu(\mu(\alpha^{-1}\beta(y),\alpha^{-1}\beta(z)))(\gamma^2\nu^{-1}(a)).
\end{align*} Similarly
\begin{align*}
 &\mu_{A\oplus V}\textbf{(}\mu_{A\oplus V}(x+a,y+b),\widetilde{\beta}(z+c)\textbf{)}\\&=\mu_{A\oplus V}\textbf{(}\mu(x,y)+\rho_\mu(x)(b)+
 \rho_\mu(\alpha^{-1}\beta(y))(\gamma\nu^{-1}(a)),\widetilde{\beta}(z+c)\textbf{)}
\\
&=\mu(\mu(x,y),\beta(z))+\rho_\mu(\mu(x,y))\circ \nu(c)+\rho_\mu(\alpha^{-1}\beta^2(z))\circ\rho_\mu(\alpha\beta^{-1}(x))(\gamma\nu^{-1}(b))+\\&
\rho_\mu(\alpha^{-1}\beta^2(z))\circ\rho_\mu(y)(\gamma^2\nu^{-2}(a)).
\end{align*}

\end{proof}

\begin{df}
Let $(A,\{\cdot,\cdot\},\mu,\alpha,\beta)$ be a BiHom-Poisson algebra, $V$ be a vector space and $ \rho_{\{\cdot,\cdot\}},\rho_{\mu}:A\longrightarrow End(V)$ be two linear maps and also $\gamma,\nu: V \longrightarrow V$ be  two linear maps. Then $(V,\rho_{\{\cdot,\cdot\}},\rho_{\mu},\gamma,\nu)$ is called a representation of $A$ if $(V,\rho_{\{\cdot,\cdot\}},\gamma,\nu)$ is a representation of $(A,\{\cdot,\cdot\},\alpha,\beta)$ and $(V,\rho_{\mu},\gamma,\nu)$ is a representation of $(A,\mu,\alpha,\beta)$ and they are compatible in the sense that for any $x,y\in A$
\begin{eqnarray}
\label{RepComp1}\rho_{\{\cdot,\cdot\}}(\mu(x, y))\nu &=& \rho_{\mu}(\beta(y))\rho_{\{\cdot,\cdot\}}(x) +\rho_{\mu}(\alpha(x))\rho_{\{\cdot,\cdot\}}(y),\\
\label{RepComp2}\rho_{\mu}(\{\beta(x),y\})\nu &=&-\rho_{\mu}(\alpha\beta(x))\rho_{\{\cdot,\cdot\}}(y)
-\rho_{\{\cdot,\cdot\}}(\beta(y))\rho_{\mu}(\alpha(x)).
\end{eqnarray}
\end{df}
\begin{thm}
 Let $(A,\{\cdot,\cdot\},\mu,\alpha,\beta)$ be a BiHom-Poisson algebra  and $(V,\rho_{\{\cdot,\cdot\}},\rho_{\mu},\gamma,\nu)$ be a representation of $A$.
Then $(A\oplus V,\mu_{A\oplus V},\{\cdot,\cdot\}_{A\oplus V},\widetilde{\alpha},\widetilde{\beta})$ is a commutative BiHom-Poisson algebra, where the maps $\mu_{A\oplus V},\{\cdot,\cdot\}_{A\oplus V},\widetilde{\alpha}$ and $\widetilde{\beta}$ are defined in  Theorem \ref{AssProduitSDirect} and  Theorem \ref{LieProduitSDirect}.
\end{thm}

\begin{proof}We need only to show that the Leibniz identity is satisfied.
  Let $x,y,z\in A$ and $a,b,c\in V$, we have
  \begin{align*}
     & \{\mu_{A\oplus V}(x+a,y+b),\widetilde{\alpha}\widetilde{\beta}(z+c)\}_{A\oplus V}-\mu_{A\oplus V}(\{x+a, \widetilde{\beta}(z+c)\}_{A\oplus V},\widetilde{\alpha}(y+b))\\&-\mu_{A\oplus V}(\widetilde{\alpha}(x+a), \{y+b,\widetilde{\alpha}(z+c)\}_{A\oplus V}) \\
      =& \{\mu(x,y)+\rho_{\mu}(x)(b)+\rho_{\mu}(\alpha^{-1}\beta(y))
      \gamma\nu^{-1}(a),\widetilde{\alpha}\widetilde{\beta}(z+c)\}_{A\oplus V} \\&-\mu_{A\oplus V}(\{x,\beta(z)\}+\rho_{\{\cdot,\cdot\}}(\alpha^{-1}\beta(x))\nu(c)    -\rho_{\{\cdot,\cdot\}}(\alpha^{-1}\beta^2(z))\gamma\nu^{-1}(a),\widetilde{\alpha}(y+b))\\&-\mu_{A\oplus V}(\widetilde{\alpha}(x+a),\{y,\alpha(z)\}+\rho_{\{\cdot,\cdot\}}(y)\gamma(c)-\rho_{\{\cdot,\cdot\}}(\beta(z))\gamma\nu^{-1}(b)) \\
      =&\{\mu(x,y),\alpha\beta(z)\}+\rho_{\{\cdot,\cdot\}}(\mu(x,y))\gamma\nu(c)
      -\rho_{\{\cdot,\cdot\}}(\beta^2(z))\gamma\nu^{-1}(\rho_{\mu}(x)(b)+\rho_{\mu}(\alpha^{-1}\beta(y))
      \gamma\nu^{-1}(a))\\
      &-\mu(\{x,\beta(z)\},\alpha(y))-\rho_{\mu}(\{x,\beta(z)\})\gamma(b)
      -\rho_{\mu}(\beta(y))\gamma\nu^{-1}(\rho_{\{\cdot,\cdot\}}(\alpha^{-1}\beta(x))\nu(c)    -\rho_{\{\cdot,\cdot\}}(\alpha^{-1}\beta^2(z))\gamma\nu^{-1}(a))
      \\&-\mu(\alpha(x),\{y,\alpha(z)\})-\rho_{\mu}(\alpha(x))(\rho_{\{\cdot,\cdot\}}(y)\gamma(c)-\rho_{\{\cdot,\cdot\}}(\alpha(z))(b))
      -\rho_{\mu}(\{\alpha^{-1}\beta(y),\beta(z)\})\gamma^2\nu^{-1}(a)\\
      =&\{\mu(x,y),\alpha\beta(z)\}+\rho_{\{\cdot,\cdot\}}(\mu(x,y))\gamma\nu(c)
      -\rho_{\{\cdot,\cdot\}}(\beta^2(z))(\rho_{\mu}(\alpha\beta^{-1}(x))\gamma\nu^{-1}(b))\\
      &-\rho_{\{\cdot,\cdot\}}(\beta^2(z))(\rho_{\mu}(y)
      \gamma^2\nu^{-2}(a))      -\mu(\{x,\beta(z)\},\alpha(y))-\rho_{\mu}(\{x,\beta(z)\})\gamma(b)
\\&      -\rho_{\mu}(\beta(y))(\rho_{\{\cdot,\cdot\}}(x)\gamma(c))    -\rho_{\mu}(\beta(y))(\rho_{\{\cdot,\cdot\}}(\beta(z))\gamma^2\nu^{-2}(a))
      -\mu(\alpha(x),\{y,\alpha(z)\})\\&-\rho_{\mu}(\alpha(x))(\rho_{\{\cdot,\cdot\}}(y)\gamma(c)
)      +\rho_{\mu}(\alpha(x))(\rho_{\{\cdot,\cdot\}}(\alpha(z))(b))
      -\rho_{\mu}(\{\alpha^{-1}\beta(y),\beta(z)\})\gamma^2\nu^{-1}(a)=0.
  \end{align*}
\end{proof}

\section{Admissible BiHom-Poisson algebras}
\label{sec:admissible}

A Poisson algebra has two binary operations, the Lie bracket and the commutative associative product. In this section we describe BiHom-Poisson algebra using only one binary operation and the twisting maps via the polarization-depolarization procedure.

\begin{df}
\label{def:admissible}
Let $(A,\mu,\alpha,\beta)$ be a BiHom-algebra.  Then $A$ is called an \textbf{admissible BiHom-Poisson algebra} if it satisfies
\begin{align}
\label{admissibility}
as_{\alpha,\beta}(\beta(x),\alpha(y),\alpha^{2}(z)) &= \frac{1}{3}\{\mu(\mu(\beta(x),\alpha\beta(z)),\alpha^2(y)) - \mu(\mu(\beta^{2}(z),\alpha(x)),\alpha^{2}(y))\nonumber\\& + \mu(\mu(\beta(y),\alpha\beta(z)),\alpha^2(x)) - \mu(\mu(\beta(y),\alpha(x))\alpha^{2}\beta(z))\},
\end{align}
for all $x,y,z \in A$, where $as_{\alpha,\beta}$ is the BiHom-associator  of $A$ defined by
\begin{equation}
\label{associator}
as_{\alpha,\beta}(x,y,z) = \mu(\mu(x,z),\beta(y)) - \mu(\alpha(x),\mu(x,z)) \end{equation}
\end{df}
If the BiHom-algebra $(A,\mu,\alpha,\beta)$ is regular then the identity  \eqref{admissibility} is equivalent to
\begin{align}
\label{admissibility2}
as_{\alpha,\beta}(x,y,z) &= \frac{1}{3}\{\mu(\mu(x,\alpha^{-1}\beta(z)),\alpha(y)) - \mu(\mu(\alpha^{-2}\beta^{2}(z),\alpha\beta^{-1}(x)),\alpha(y))\nonumber \\&+ \mu(\mu(\alpha^{-1}\beta(y),\alpha^{-1}\beta(z)),\alpha^{2}\beta^{-1}(x)) - \mu(\mu(\alpha^{-1}\beta(y),\alpha\beta^{-1}(x))\beta(z))\}.
\end{align}

\begin{prop}
  Let $(A,\mu)$ be an admissible Poisson algebra and $\alpha,\beta:A\rightarrow A$ two commuting Poisson algebra morphisms. Then $(A,\mu_{\alpha,\beta}=\mu\circ(\alpha\otimes\beta),\alpha,\beta)$ is an admissible BiHom-Poisson algebra.
  \end{prop}
  \begin{proof}
Let $x,y,z\in A$
\small{\begin{align*}
  &\mu_{\alpha,\beta}(\mu_{\alpha,\beta}(\beta(x),\alpha(y)),\alpha^{2}\beta(z))
  -\mu_{\alpha,\beta}(\alpha\beta(x),\mu_{\alpha,\beta}(\alpha(y),\alpha^{2}(z)))- \frac{1}{3}\{\mu_{\alpha,\beta}(\mu_{\alpha,\beta}(\beta(x),\alpha\beta(z)),\alpha^2(y))\\& - \mu_{\alpha,\beta}(\mu_{\alpha,\beta}(\beta^{2}(z),\alpha(x)),\alpha^{2}(y)) + \mu_{\alpha,\beta}(\mu_{\alpha,\beta}(\beta(y),\alpha\beta(z)),\alpha^2(x)) - \mu_{\alpha,\beta}(\mu_{\alpha,\beta}(\beta(y),\alpha(x))\alpha^{2}\beta(z))\}\\
  =&\mu(\mu(\alpha^2\beta(x),\alpha^2\beta(y)),\alpha^{2}\beta^2(z))
  -\mu(\alpha^2\beta(x),\mu(\alpha^2\beta(y),\alpha^{2}\beta^2(z)))- \frac{1}{3}\{\mu(\mu(\alpha^2\beta(x),\alpha^2\beta^2(z)),\alpha^2\beta(y))\\& - \mu(\mu(\alpha^2\beta^{2}(z),\alpha^2\beta(x)),\alpha^{2}\beta(y)) + \mu(\mu(\alpha^2\beta(y),\alpha^2\beta^2(z)),\alpha^2\beta(x)) - \mu(\mu(\alpha^2\beta(y),\alpha^2\beta(x))\alpha^{2}\beta^2(z))\}\\
  =0
\end{align*}}
  \end{proof}

As usual in \eqref{admissibility} the product $\mu$ is denoted by juxtapositions of elements in $A$.  An admissible BiHom-Poisson algebra with $\alpha=\beta = Id$ is exactly an \textbf{admissible Poisson algebra} as defined in \cite{gr}.

To compare BiHom-Poisson algebras and admissible BiHom-Poisson algebras, we need the following function, which generalizes a  similar function in \cite{mr}.

\begin{df}
Let $(A,\mu,\alpha,\beta)$ be a regular BiHom-algebra.  Define the quadruple
\begin{equation}
\label{pa}
P(A) = \left(A, \{\cdot,\cdot\} , \bullet , \alpha, \beta\right),
\end{equation}
where $\{x,y\} = \frac{1}{2}(\mu(x,y) - \mu(\alpha^{-1}\beta(y)\alpha\beta^{-1}(x)))$ and $x\bullet y = \frac{1}{2}(\mu(x,y) +\mu(\alpha^{-1}\beta(y)\alpha\beta^{-1}(x)))$
called the \textbf{polarization} of $A$.  We call $P$ the \textbf{polarization function}.
\end{df}

The following result says that admissible BiHom-Poisson algebras, and only these BiHom-algebras, give rise to BiHom-Poisson algebras via polarization.

\begin{thm}
\label{thm:polar}
Let $(A,\mu,\alpha,\beta)$ be a regular BiHom-algebra.  Then the polarization $P(A)$ is a regular BiHom-Poisson algebra if and only if $A$ is an admissible BiHom-Poisson algebra.
\end{thm}

\begin{proof}
   For any $x, y, z\in A$, we will check that the $(A, \{\cdot,\cdot\},\alpha,\beta)$ is a BiHom-Lie algebra. Indeed, we have
   \begin{align*}
\{\beta(x),\alpha(y)\}=\beta(x)\alpha(y)-\beta(y)\alpha(x)=-\{\beta(y),\alpha(x)\},
\end{align*}
so the antisymmetry of $\{\cdot,\cdot\}$ is satisfied. Now, we verify the BiHom-Jacobi identity
\begin{align*}
&\{\beta^2(x), \{\beta(y), \alpha(z)\}\}+\{\beta^2(y), \{\beta(z), \alpha(x)\}\}+\{\beta^2(z), \{\beta(x), \alpha(y)\}\}\\
&=\{\beta^2(x), \frac{1}{2}(\beta(y)\cdot \alpha(z)-\beta(z)\cdot \alpha(y))\}+
\{\beta^2(y), \frac{1}{2}(\beta(z)\cdot \alpha(x)-\beta(x)\cdot \alpha(z))\}\\
&+\{\beta^2(z), \frac{1}{2}(\beta(x)\cdot \alpha(y)-\beta(y)\cdot \alpha(x))\}\\
&=\frac{1}{4}\Big(-as_{\alpha\beta}(\beta(x), \beta(y), \alpha(z))+as_{\alpha\beta}(\alpha^{-1}\beta^{2}(x), \beta(z), \alpha(y))\\
&-as_{\alpha\beta}(\alpha^{-1}\beta^{2}(y), \beta(z), \alpha(x))+as_{\alpha\beta}(\alpha^{-1}\beta^{2}(z), \beta(y), \alpha(x))\\
&+as_{\alpha\beta}(\alpha^{-1}\beta^{2}(y),\beta(x), \alpha(z))-as_{\alpha\beta}(\alpha^{-1}\beta^{2}(z), \beta(x), \alpha(y))\Big)=0.
\end{align*}
Next, we check that $(A, \bullet,\alpha,\beta)$ is a BiHom-commutative BiHom-associative algebra. Indeed, for any $x, y, z\in A $, the prove of BiHom-commutativity of $\mu$ is similar to the BiHom-skewsymmetry of $\{\cdot,\cdot\}$ checked before.
\begin{align*}
&(x\bullet y)\bullet \beta(z)-\alpha(x)\bullet(y\bullet z)=\frac{1}{2}(\mu(x,y)
-\mu(\alpha^{-1}\beta(y), \alpha\beta^{-1}(x)))\bullet \beta(z)
-\alpha(x)\bullet\frac{1}{2}(\mu(y,z)\\&-\mu(\alpha^{-1}\beta(z), \alpha\beta^{-1}(y)))=\frac{1}{4}\Big(as_{\alpha\beta}(x, y, z)-as_{\alpha\beta}(\alpha^{-2}\beta^{2}(z), y, \alpha^{2}\beta^{-2}(x))\\
&+as_{\alpha\beta}(\alpha^{-2}\beta^{2}(z), \alpha\beta^{-1}(x), \alpha\beta^{-1}(y))+\mu(\mu(\alpha^{-2}\beta^{2}(z),\alpha\beta^{-1}(x)),\alpha(y))\\
&-\mu(\mu(\alpha^{-1}\beta(y),\alpha\beta^{-1}(x)),\beta(z))+\mu(\mu(\alpha^{-1}\beta(y),\alpha^{-1}\beta(z)),\alpha^{2}\beta^{-1}(x))\\
&-as_{\alpha\beta}(x,\alpha^{-1}\beta(y), \alpha\beta^{-1}(z))-\mu(\mu(x,\alpha^{-1}\beta(y)),\alpha(z))=0.
\end{align*}
Finally, we check the condition :\ \ \
$\{x\bullet y,\alpha\beta (z)\}-\{x, \beta(z)\}\bullet \alpha(y)-\alpha(x)\bullet\{y,\alpha(z)\}.$ Indeed, we have
\begin{align*}
&\{x\bullet y,\alpha\beta (z)\}-\{x, \beta(z)\}\bullet \alpha(y)-\alpha(x)\bullet\{y,\alpha(z)\}\\
&=\frac{1}{4}\Big(as_{\alpha\beta}(x, y, \alpha(z))-as_{\alpha\beta}(x, \beta(z), \alpha\beta^{-1}(y))-as_{\alpha\beta}(\alpha^{-1}\beta(y),\beta( z), \alpha^{2}\beta^{-2}(x))+as(\alpha^{-1}\beta^{2}(z), y, \alpha^{2}\beta^{-2}(x))\\
&+as_{\alpha\beta}(\alpha^{-1}\beta^{2}(z), \alpha\beta^{-1}( x), \alpha\beta^{-1}(y))+as_{\alpha\beta}(\alpha^{-1}\beta(y),\alpha\beta^{-1} (x), \alpha(z))\Big)=0.
\end{align*}
The proof is finished.
\end{proof}

The following result says that there is a bijective correspondence between admissible BiHom-Poisson algebras and BiHom-Poisson algebras via polarization and depolarization. 

\begin{cor}\label{cor:polar}
Let $(A,\{\cdot,\cdot\},\bullet,\alpha,\beta)$ be a  BiHom-Poisson algebra.  Define the BiHom-algebra
\begin{equation}
\label{pminusa}
P^-(A) = \left(A,\mu = \{\cdot,\cdot\}+\bullet, \alpha,\beta\right),
\end{equation}
then $P^-(A)$ is an admissible BiHom-Poisson algebra
called the \textbf{depolarization} of $A$.  We call $P^-$ the \textbf{depolarization function}.
\end{cor}


\begin{proof}
If $(A,\mu,\alpha)$ is a regular admissible BiHom-Poisson algebra, then $P(A)$ is a BiHom-Poisson algebra by Theorem \ref{thm:polar}.  We have $P^-(P(A)) = A$ because
$$
\mu(x,y) = \frac{1}{2}(\mu(x,y) - \mu(\alpha^{-1}\beta(y),\alpha\beta^{-1}(x))) + \frac{1}{2}(\mu(x,y) + \mu(\alpha^{-1}\beta(y),\alpha\beta^{-1}(x)), \ \forall \ x,\ y\in A.
$$\end{proof}
\section{Classification of BiHom-Poisson algebras}

Let $(A,\{\cdot,\cdot\},\mu,\alpha,\beta)$ be a  BiHom-Poisson algebra, in this section, we provide a list of 2-dimensional BiHom-Poisson algebras, where the morphisms $\alpha$ and $\beta$ are diagonal.

\begin{table}[h]
\begin{tabular}{|l||p{4cm}||p{5cm}||p{3cm}|}
\hline
 Algebras & Multiplications & Brackets & Morphisms \\ \hline
$Alg_{1} $ & $\mu(e_{1},e_{1})=c_{11}^{1}e_{1}$,\par $\mu(e_{2},e_{2})=c_{22}^{2}e_{2}$, & $\{ e_{1},e_{1}\}=d_{11}^{1}e_{1}$,\par  $\{ e_{2},e_{1}\}=d_{21}^{1}e_{1}$,
& $\alpha = \left(
\begin{array}{cc}
0 & 0 \\
0 & 1%
\end{array}
\right),$ \par$\beta = \left(
\begin{array}{cc}
0 & 0 \\
0 & 1%
\end{array}
\right).$ \\ \hline
$Alg_{2}$ &  $
\mu(e_{1},e_{1})=c_{11}^{1}e_{1},$ &  $\{e_{1},e_{1}\}=d_{11}^{1}e_{1},$ \par $\{e_{2},e_{1}\}=e_{1}$, & $\alpha =
\left(
\begin{array}{cc}
0 & 0 \\
0 & a_{22}%
\end{array}
\right),$ \par $\beta = \left(
\begin{array}{cc}
0 & 0 \\
0 & b_{22}%
\end{array}
\right).$ \\ \hline
$Alg_{3}$ &  $%
\mu(e_{1},e_{1})=c_{11}^{1}e_{1}$,\par $\mu(e_{2},e_{2})=c_{22}^{2}e_{2},$ &  $\{ e_{1},e_{2}\}=d_{12}^{2}e_{2}$,\par $\{ e_{2},e_{2}\}=d_{22}^{2}e_{2},$ & $\alpha = \left(
\begin{array}{cc}
1 & 0 \\
0 & 0%
\end{array}
\right),$\par $\beta = \left(
\begin{array}{cc}
1 & 0 \\
0 & 0%
\end{array}
\right).$ \\ \hline
$Alg_{4}$ & $\mu(e_{2},e_{2})=c_{22}^{2}e_{2}$,  & $\{e_{1},e_{2}\}=e_{2}$, \par$\{e_{2},e_{2}\}=d_{22}^{2}e_{2}$, & $\alpha = \left(
\begin{array}{cc}
a_{11} & 0 \\
0 & 0%
\end{array}
\right),$ \par$\beta = \left(
\begin{array}{cc}
b_{11} & 0 \\
0 & 0%
\end{array}
\right).$ \\ \hline
$Alg_{5}$ &  $%
\mu(e_{2},e_{2})=c_{22}^{2}e_{2},$ &  $\{e_{1},e_{1}\}=e_{1}$,\par$\{e_{1},e_{2}\}=d_{12}^{1}e_{1}$,\par$\{e_{2},e_{1}\}=d_{21}^{1}e_{1}$,  & $\alpha = \left(
\begin{array}{cc}
0 & 0 \\
0 & 1%
\end{array}
\right),$\par $\beta = \left(
\begin{array}{cc}
0 & 0 \\
0 & 1%
\end{array}
\right).$ \\ \hline

$Alg_{6}$ &  $\mu(e_{2},e_{1})= c_{21}^{1}e_{1}$, \par $\mu(e_{2},e_{2})= \frac{c_{21}^{1}}{b_{11}}e_{2}$,  & $\{ e_{2},e_{1}\}=d_{21}^{1}e_{1}$,  & $\alpha = \left(
\begin{array}{cc}
0 & 0 \\
0 & 1%
\end{array}
\right),$ \par$\beta = \left(
\begin{array}{cc}
b_{11} & 0 \\
0 & 1%
\end{array}
\right).$ \\ \hline
$Alg_{7}$ & $\mu(e_{1},e_{1})=c_{11}^{1}e_{1}$, & $\{ e_{1},e_{2}\}=d_{12}^{2}e_{2}$,\par $\{ e_{2},e_{1}\}=d_{21}^{2}e_{2}$, \par$\{ e_{2},e_{2}\}=e_{2}$, & $\alpha = \left(
\begin{array}{cc}
1 & 0 \\
0 & 0%
\end{array}
\right),$\par $\beta = \left(
\begin{array}{cc}
1 & 0 \\
0 & 0%
\end{array}
\right).$ \\ \hline
$Alg_{8}$ & $\mu(e_{1},e_{1})=c_{11}^{1}e_{1}$,  &  $\{e_{1},e_{2}\}=d_{12}^{2}e_{2}$,  & $\alpha = \left(
\begin{array}{cc}
1 & 0 \\
0 & 0%
\end{array}
\right),$\par $\beta = \left(
\begin{array}{cc}
1 & 0 \\
0 & b_{22}%
\end{array}
\right).$ \\ \hline
$Alg_{9}$ & $\mu(e_{1},e_{1})=c_{11}^{1}e_{1}$, \par $\mu(e_{1}, e_{2})=c_{11}^{1}b_{22}e_{2}$,& $\{e_{1},e_{2}\}=d_{12}^{2}e_{2}$,   & $\alpha = \left(
\begin{array}{cc}
1 & 0 \\
0 & 0%
\end{array}
\right),$\par $\beta = \left(
\begin{array}{cc}
1 & 0 \\
0 & b_{22} %
\end{array}
\right).$ \\ \hline
$Alg_{10}$ & $\mu(e_{1},e_{1})=c_{11}^{1}e_{1}$,  & $\{e_{2},e_{1}\}=d_{21}^{2}e_{2}$, & $\alpha = \left(
\begin{array}{cc}
1 & 0 \\
0 & a_{22}%
\end{array}
\right),$ \par$\beta = \left(
\begin{array}{cc}
1 & 0 \\
0 & 0 %
\end{array}
\right).$ \\ \hline
\end{tabular}%
\end{table}
 \clearpage

\begin{table}[h]
\begin{tabular}{|l||p{4cm}||p{5cm}||p{3cm}|}
\hline

$Alg_{11} $ & $\mu(e_{1},e_{1})=c_{11}^{1}e_{1}$, \par $\mu(e_{2},e_{1})=c_{11}^{1}a_{22}e_{2}$, & $\{ e_{2},e_{1}\}=d_{21}^{2}e_{2}$,
& $\alpha = \left(
\begin{array}{cc}
1 & 0 \\
0 & a_{22}%
\end{array}
\right),$ \par$\beta = \left(
\begin{array}{cc}
1 & 0 \\
0 & 0%
\end{array}
\right).$ \\ \hline
$Alg_{12}$ &  $%
\mu(e_{2},e_{2})=c_{22}^{2}e_{2},$ &   $\{e_{1},e_{2}\}=d_{12}^{1}e_{1}$, & $\alpha =
\left(
\begin{array}{cc}
a_{11} & 0 \\
0 & 1%
\end{array}
\right),$ \par $\beta = \left(
\begin{array}{cc}
0 & 0 \\
0 & 1%
\end{array}
\right).$ \\ \hline
$Alg_{13}$ &  $%
\mu(e_{1},e_{2})=c_{12}^{1}e_{1}$,\par $\mu(e_{2},e_{2})=\frac{c_{12}^{1}}{a_{11}}e_{2},$ &  $\{ e_{1},e_{2}\}=d_{12}^{1}e_{1}$, & $\alpha = \left(
\begin{array}{cc}
a_{11} & 0 \\
0 & 1%
\end{array}
\right),$ \par $\beta = \left(
\begin{array}{cc}
0 & 0 \\
0 & 1%
\end{array}
\right).$ \\ \hline
$Alg_{14}$ & $\mu(e_{2},e_{2})=c_{22}^{2}e_{2}$,  & $\{e_{1},e_{1}\}=e_{1}$,  \par$\{e_{2},e_{1}\}=d_{21}^{1}e_{1}$, & $\alpha = \left(
\begin{array}{cc}
0 & 0 \\
0 & 1%
\end{array}
\right),$ \par $\beta = \left(
\begin{array}{cc}
1 & 0 \\
0 & 1%
\end{array}
\right).$ \\ \hline
$Alg_{15}$ &  $\mu(e_{2},e_{1})=c_{21}^{1}e_{1}$,\par $\mu(e_{2},e_{2})=c_{21}^{1}e_{2},$ &  $\{e_{1},e_{1}\}=e_{1}$,\par $\{e_{2},e_{1}\}=d_{21}^{1}e_{1}$,  & $\alpha = \left(
\begin{array}{cc}
0 & 0 \\
0 & 1%
\end{array}
\right),$ \par $\beta = \left(
\begin{array}{cc}
1 & 0 \\
0 & 1%
\end{array}
\right).$ \\ \hline

$Alg_{16}$ &  \par $\mu(e_{2},e_{2})= c_{22}^{2}e_{2}$,  & $\{ e_{2},e_{1}\}=d_{21}^{1}e_{1}$,  & $\alpha = \left(
\begin{array}{cc}
0 & 0 \\
0 & 1%
\end{array}
\right),$ \par $\beta = \left(
\begin{array}{cc}
b_{11} & 0 \\
0 & 1%
\end{array}
\right).$ \\ \hline
$Alg_{17}$ & $\mu(e_{1},e_{1})=c_{11}^{1}e_{1}$, & $\{ e_{1},e_{2}\}=d_{12}^{2}e_{2}$,\par $\{ e_{2},e_{2}\}=e_{2}$, & $\alpha = \left(
\begin{array}{cc}
1 & 0 \\
0 & 0%
\end{array}
\right),$ \par $\beta = \left(
\begin{array}{cc}
1 & 0 \\
0 & 1%
\end{array}
\right).$ \\ \hline
$Alg_{18}$ & $\mu(e_{1},e_{1})=c_{11}^{1}e_{1}$,\par $\mu(e_{1},e_{2})=c_{11}^{1}e_{2}$, &  $\{e_{1},e_{2}\}=d_{12}^{2}e_{2}$,\par $\{e_{2},e_{2}\}=e_{2}$,  & $\alpha = \left(
\begin{array}{cc}
1 & 0 \\
0 & 0%
\end{array}
\right),$ \par $\beta = \left(
\begin{array}{cc}
1 & 0 \\
0 & 1%
\end{array}
\right).$ \\ \hline
$Alg_{19}$ & $\mu(e_{2},e_{2})=c_{22}^{2}e_{2}$,  & $\{e_{1},e_{1}\}=e_{1}$, \par $\{e_{1},e_{2}\}=d_{12}^{1}e_{1}$,   & $\alpha = \left(
\begin{array}{cc}
1 & 0 \\
0 & 1%
\end{array}
\right),$ \par $\beta = \left(
\begin{array}{cc}
0 & 0 \\
0 & 1 %
\end{array}
\right).$ \\ \hline
$Alg_{20}$ & $\mu(e_{1},e_{1})=c_{11}^{1}e_{1}$,  & $\{e_{2},e_{1}\}=d_{21}^{2}e_{2}$,\par $\{e_{2},e_{2}\}=e_{2}$, & $\alpha = \left(
\begin{array}{cc}
1 & 0 \\
0 & 1%
\end{array}
\right),$ \par $\beta = \left(
\begin{array}{cc}
1 & 0 \\
0 & 0 %
\end{array}
\right).$ \\ \hline
\end{tabular}%
\end{table}


\clearpage
\begin{thebibliography}{10}
\bibitem{AmriMakhlouf} H. Amri, A. Makhlouf,  Non-commutative ternary Nambu-Poisson algebras and ternary Hom-Nambu-Poisson algebras. Journal of Generalized Lie Theory and Applications, 9 :221 (2015). 
\bibitem{GDito} G. Dito, M. Flato, D. Sternheimer, and L. Takhtajan, Deformation quantization and Nambu mechanics.
Commun. Math. Phys., \textbf{183}, 1--22, (1997).
\bibitem{GDitoF} G. Dito, M. Flato, and D. Sternheimer. Nambu mechanics, n-ary operations and their
quantization.In Deformation Theory and Symplectic Geometry. Math. Phys. Stud. \textbf{20}, Kluwer Academic
Publishers, Dordrecht, Boston, 43--66, (1997).
\bibitem{f:nliealgebras} V.T. Filippov, $n$-Lie algebras,(Russian), Sibirsk. Mat. Zh. \textbf{26}, no. 6, 126--140 (1985).
(English translation: Siberian Math. J. \textbf{26}, no. 6, 879--891, (1985))

\bibitem{f:nliealgebras} V.T. Filippov, $n$-Lie algebras,(Russian), Sibirsk. Mat. Zh. \textbf{26}, no. 6, 126--140 (1985).
(English translation: Siberian Math. J. \textbf{26}, no. 6, 879--891, (1985))

\bibitem{gr}
M. Goze and E. Remm, Poisson algebras in terms of non-associative algebras, J. Alg. 320 (2008) 294--317.

\bibitem{GrazianiMakhloufMeniniPanaite}
G. Graziani, A. Makhlouf, C. Menini, F. Panaite, BiHom-associative algebras, BiHom-Lie algebras and BiHom-bialgebras. Symmetry, Integrability and Geometry : Methods and Applications SIGMA 11 (2015), 086, 34 pages



\bibitem{luimakhlouf}
L. Liu, A. Makhlouf, C. Menini, F. Panaite,  BiHom-pre-Lie algebras, BiHom-Leibniz algebras and Rota-Baxter operators on BiHom-Lie algebras, arXiv:1706.00474. (2017).


\bibitem{luimakhlouf2}
L. Liu, A. Makhlouf, C. Menini, F. Panaite,  Rota-Baxter operators on BiHom-associative algebras and related structures.
Colloq. Math., Vol. 161 (2020), 263--294. 

\bibitem{LMMP20}
L. Liu, A. Makhlouf, C. Menini, F. Panaite, Tensor products and perturbations of 
BiHom-Novikov-Poisson algebras. Preprint 2020.

\bibitem{kubo}
F. Kubo, Finite-dimensional non-commutative Poisson algebras, J. Pure Appl. Alg. 113 (1996) 307--314.

\bibitem{ms}
A. Makhlouf and S. Silvestrov, Hom-algebra structures, J. Gen. Lie Theory Appl. 2 (2008) 51--64.

\bibitem{ms2}
A. Makhlouf and S. Silvestrov, Hom-algebras and Hom-coalgebras, J. Alg. Appl. 9 (2010) 1--37.


\bibitem{mr}
M. Markl and E. Remm, Algebras with one operation including Poisson and other Lie-admissible algebras, J. Alg. 299 (2006) 171--189.

\bibitem{n:generalizedmech} Y. Nambu, Generalized Hamiltonian mechanics, Phys. Rev. D (3), \textbf{7}, 2405--2412, (1973).

\bibitem{ss}
P. Schaller and T. Strobl, Poisson structure induced (topological) field theories, Mod. Phys. Lett. A 9 (1994) 3129--3136.

\bibitem{vaisman}
I. Vaisman, Lectures on the geometry of Poisson manifolds, Birkh\"{a}user, Basel, 1994.

\bibitem{Yau:Noncomm} D. Yau, Non-commutative Hom-Poisson algebras, arXiv:1010.3408 v1 [math.RA], (2010).

\bibitem{RepBiHomLie}
C. Yongsheng, and H. Qi, Representations of Bihom-Lie algebras. arXiv:1610.04302 (2016).
\end{thebibliography}
\end{document}